\documentclass[conference]{IEEEtran}
\IEEEoverridecommandlockouts
\usepackage{cite}

\usepackage[T1]{fontenc}
\usepackage[latin9]{inputenc}
\usepackage{geometry}
\usepackage[english]{babel}
\usepackage{amsthm}
\PassOptionsToPackage{normalem}{ulem}
\usepackage{ulem}

\usepackage{tikz}
\usetikzlibrary{shapes,arrows}

\usepackage{wrapfig}
\usepackage{amsmath,amssymb,amsfonts}
\usepackage{algorithm}
\usepackage{algorithmic}
\usepackage{graphicx}
\usepackage[all,cmtip]{xy}
\usepackage{textcomp}
\usepackage{cleveref}
\usepackage{xcolor}
\def\BibTeX{{\rm B\kern-.05em{\sc i\kern-.025em b}\kern-.08em
    T\kern-.1667em\lower.7ex\hbox{E}\kern-.125emX}}
    
\newtheorem{theorem}{Theorem}

\newtheorem{lemma}{Lemma}

\numberwithin{equation}{section}
\numberwithin{theorem}{section}
\numberwithin{lemma}{section}
\numberwithin{corollary}{section}
\numberwithin{definition}{section}
\numberwithin{example}{section}
\numberwithin{remark}{section}
\numberwithin{property}{section}
\numberwithin{proposition}{section}

\begin{document}

\title{On Cyclic Finite-State Approximation of Data-Driven Systems\\
{\footnotesize 
}
}

\author{\IEEEauthorblockN{Fredy Vides}
\IEEEauthorblockA{\textit{Scientific Computing Innovation Center} \\
\textit{Department of Applied Mathematics}\\
\textit{Universidad Nacional Aut\'onoma de Honduras, UNAH}\\
Honduras \\
fredy.vides@unah.edu.hn}\\
\IEEEauthorblockA{\textit{Graduate Program in Structural Engineering} \\
\textit{Universidad Tecnol\'ogica Centroamericana, UNITEC}\\
Honduras \\
vides@unitec.edu}
}

\maketitle

\begin{abstract}
In this document, some novel theoretical and computational techniques for constrained approximation of data-driven systems, are presented. The motivation for the development of these techniques came from structure-preserving matrix approximation problems that appear in the fields of system identification and model predictive control, for data-driven systems and processes. The research reported in this document is focused on finite-state approximation of data-driven systems. 

Some numerical implementations of the aforementioned techniques in the simulation and model predictive control of some generic data-driven systems, that are related to electrical signal transmission models, are outlined.
\end{abstract}

\begin{IEEEkeywords}
Closed-loop control system, state transition matrix, singular value decomposition, structured matrices, pseudospectrum.
\end{IEEEkeywords}

\maketitle

\section{Introduction}

The purpose of this document is to present some novel theoretical and computational techniques for constrained approximation of data-driven systems. These systems can be interpreted as discrete-time systems that can be {\em partially} described by difference equations of the form.
\begin{equation}
\Sigma: \left\{
\begin{array}{l}
x_{t+1}=F(x_t,t),~~ t\geq 1\\
x_1\in \Sigma\subseteq \mathbb{C}^n
\end{array}
\right.
\label{data_driven_sys_def}
\end{equation}
where $\Sigma\subseteq\mathbb{C}^n$ is the set of {\em valid} states for the system, and where $F:\mathbb{C}^n\times \mathbb{R}\to \mathbb{C}^n$ is some constrained map that is either partially known, or needs to be determined/discovered based on some (sampled) data $\{x_t\}_{1\leq t\leq N}$, obtained in the form of data {\em snapshots} related to the system $\Sigma$ under study. One can also interpret the map $F$ in \eqref{data_driven_sys_def} as a {\em black-box device}, that needs to be determined in such a way that it can be used to transform the {\em present state} $x_t$ into the {\em next state} $x_{t+1}$, according to \eqref{data_driven_sys_def}.

Since in this study, the information about a given system is provided essentially by orbits ({\em data sequences}) in some valid state space $\Sigma$, from here on, we will refer to data-driven systems in terms of sets or elements in a state space $\Sigma$.

The discovery, simulation and predictive control of the evolution laws for systems of the form \eqref{data_driven_sys_def} are highly important in data-based analytics and forcasting, for models related to the automatic control of systems and processes in industry and engineering. 

Although, on this paper we will focus on the solution of the theoretical problems related to the existence and computability of finite-state approximation of data-driven systems determined by data sequences described by \eqref{data_driven_sys_def}, the constructive nature of the procedures presented in this document allows one to derive prototypical algorithms like the one presented in \S\ref{section_algorithms}. Some numerical implementations of this prototypical algorithm are presented in \S\ref{section_experiments}.

Given an orbit $\{x_t\}_{t\geq 1}$ of a data-driven system $\Sigma$ determined by \eqref{data_driven_sys_def}, we will approach the computation of finite-state approximations of the state transition matrices $\{\mathcal{T}_{t,s}\}_{t,s\geq 1}$ that satisfy the equations $\mathcal{T}_{t,s}x_t=x_{t+s}$, by computing a closed-loop control system $\hat{\Sigma}$ that is determined by the decomposition
\begin{equation}
\hat{\Sigma}:\left\{
\begin{array}{l}
\hat{x}_{1}=\mathcal{L}x_1\\
\hat{x}_{t+1}=\hat{\mathcal{T}}_t\hat{x}_t\\
x_{t}=\mathcal{K}\hat{x}_{t}
\end{array}
\right.,t\geq 1
\label{CLS_Rep}
\end{equation}
related to some available sampled data $\{\tilde{x}_t\}_{1\leq t\leq N}\subseteq \Sigma$, with $\hat{\Sigma}\subseteq \{\tilde{x}_t\}_{1\leq t\leq N}$ and where the matrices $\mathcal{K},\mathcal{L}$, $\{\hat{\mathcal{T}}_t\}_{t\geq 1}$ need to be determined based on the sampled data. 

Once a system like \eqref{CLS_Rep} has been computed, one can ({\em approximately}) describe or predict the behavior of $\Sigma$ using the state transition matrices $\hat{\mathcal{T}}_{t}$. The theoretical and computational aspects of the corresponding procedures are studied in \S\ref{section_cyclic_FSA} and \S\ref{section_algorithms}.

\section{Cyclic Finite-State Approximation}

\label{section_cyclic_FSA}

\subsection{Notation}
We will write $\mathcal{S}^1$ to denote the unit circle in $\mathbb{C}$ that is determined by the expression $\{z\in \mathbb{C}~|~|z|=1\}$.

We will write $\mathbb{Z}^+$ to denote the set of positive integers, and $\mathbf{1}_n$ and $\mathbf{0}_{n}$ to denote the identity and zero matrices in $\mathbb{C}^{n\times n}$, respectively. From here on, given a matrix $X\in \mathbb{C}^{m\times n}$, we will write $X^\ast$ to denote the conjugate transpose of $X$ determined by $X^\ast=\overline{X^\top}=[\overline{X}_{ji}]$ in $\mathbb{C}^{n\times m}$. We will represent vectors in $\mathbb{C}^n$ as column matrices in $\mathbb{C}^{n\times 1}$.

Given two positive integers $p,q$ such that $p\geq q$, we will write $p~\mathrm{mod}~q$ to denote the smallest integer $0\leq r<q$ such that $p=mq+r$, for some integer $m$.

Given a matrix $Z\in \mathbb{C}^{n\times n}$, and a polynomial $p\in \mathbb{C}[z]$ over the complex numbers determined by the expression $p(z)=a_0+a_1z+\cdots+a_mz^m$, we will write $p(Z)$ to denote the matrix in $\mathbb{C}^{n\times n}$ defined by the expression $p(Z)=a_0\mathbf{1}_n+a_1Z+\cdots+a_mZ^m$, and we will write $\mathbb{C}[Z]$ to denote the matrix set $\{q(Z)|q\in \mathbb{C}[z]\}$.

We will write $\|\cdot\|_2$ to denote the Euclidean norm in $\mathbb{C}^n$ determined by $\|x\|_2=\sqrt{x^\ast x}=(\sum_{j=1}^n |x_j|^2)^{1/2}$, $x\in \mathbb{C}^n$. Given $\varepsilon>0$ and $A\in\mathbb{C}^{n\times n}$, we will write $\sigma_\varepsilon(A)$ to denote the $\varepsilon$-pseudospectrum of $A$, that by \cite[Theorem 2.1]{bookPspectra} is equivalent to the set of $z\in \mathbb{C}$ such that
\begin{equation}
\|(z\mathbf{1}_n-A)v\|_2<\varepsilon
\label{pspectra_def}
\end{equation}
for some $v\in\mathbb{C}^n$ with $\|v\|_2=1$. 

In this document we will write $\hat{e}_{j,n}$ to denote the matrices in $\mathbb{C}^{n\times 1}$ representing the canonical basis of $\mathbb{C}^{n}$ (the $j$-column of the $n\times n$ identity matrix), that are determined by the expression
\begin{equation}
\hat{e}_{j,n}=
\begin{bmatrix}
\delta_{1,j} & \delta_{2,j} & \cdots & \delta_{n-1,j}  & \delta_{n,j}
\end{bmatrix}^\top
\label{ej_def}
\end{equation}
for each $1\leq j\leq n$, where $\delta_{k,j}$ is the Kronecker delta determined by the expression.
\begin{equation}
\delta_{k,j}=
\left\{
\begin{array}{l}
1, \:\: k=j\\
0, \:\: k\neq j
\end{array}
\right.
\label{delta_def}
\end{equation}

\subsection{Generic Cyclic Shift Matrices}

Let us consider the matrix $C_{k,n}\in \mathbb{C}^{n\times n}$ determined by the expression.
\begin{equation}
C_{k,n}=
\begin{bmatrix}
0 & 0 & 0 & \cdots & \delta_{k,1}\\
1 & 0 & 0 & \cdots & \delta_{k,2}\\
0 & 1 & 0 & \cdots & \delta_{k,3}\\
\vdots & \ddots & \ddots & \ddots & \vdots\\
0 & \cdots & 0 & 1 & \delta_{k,n}
\end{bmatrix}
\label{c_shift_first_rep}
\end{equation}
We call $C_{k,n}\in \mathbb{C}^{n\times n}$ a {\bf \em Generic Cyclic Shift} matrix or GCS in this document. It can be seen that a matrix $C_{k,n}\in \mathbb{C}^{n\times n}$ determined by \eqref{c_shift_first_rep} can be represented in the form.
\begin{equation}
C_{k,n}=\hat{e}_{k,n}\hat{e}_{n,n}^\ast+\sum_{j=1}^{n-1}\hat{e}_{j+1,n}\hat{e}_{j,n}^\ast
\label{c_shift_first_sum_rep}
\end{equation}

\begin{lemma}
\label{shift_property}
The GCS matrix $C_{k,n}\in \mathbb{C}^{n\times n}$ satisfies the following conditions.
\begin{equation}
C_{k,n}\hat{e}_{j,n}=
\left\{
\begin{array}{l}
\hat{e}_{j+1,n}, \:\: 1\leq j\leq n-1\\
\hat{e}_{k,n}, \:\: j=n
\end{array}
\right.
\label{Sn_consntraints}
\end{equation}
\end{lemma}
\begin{proof}
By \eqref{delta_def} and \eqref{c_shift_first_sum_rep} we will have that.
\begin{eqnarray*}
C_{k,n}\hat{e}_{j,n}&=&\hat{e}_{k,n}(\hat{e}^\ast_{n,n}\hat{e}_{j,n})+\sum_{s=1}^{n-1}\hat{e}_{s+1,n}(\hat{e}_{s,n}^\ast\hat{e}_{j,n} )\\
&=&\delta_{n,j}\hat{e}_{k,n}+\sum_{s=1}^{n-1}\delta_{s,j}\hat{e}_{s+1,n}\\
&=&\left\{
\begin{array}{l}
\hat{e}_{j+1,n}, \:\: 1\leq j\leq n-1\\
\hat{e}_{k,n}, \:\: j=n
\end{array}
\right.
\end{eqnarray*}
This completes the proof.
\end{proof}

It can be seen that the GCS matrix $C_{k,n}$ determined by \eqref{c_shift_first_rep} can be expressed in the form.
\begin{equation}
C_{k,n}=
\begin{bmatrix}
\hat{e}_{2,n} & \hat{e}_{3,n} & \hat{e}_{4,n} & \cdots & \hat{e}_{k,n}
\end{bmatrix}
\label{c_shift_second_rep}
\end{equation}

By \eqref{c_shift_second_rep} and elementary linear algebra we will have that $C_{k,n}$ is the companion matrix of the polynomial $p_k\in \mathbb{C}[z]$ determined by the expression. 
\begin{equation}
p_k(z)=z^n-z^{k-1}
\label{Ck_min_poly}
\end{equation}
This means that each GCS $C_{k,n}\in\mathbb{C}^{n\times n}$ satisfies the equation.
\begin{equation}
p_k(C_{k,n})=C_{k,n}^n-C_{k,n}^{k-1}=\mathbf{0}_{n}
\label{Ck_min_poly_cond}
\end{equation}

\begin{theorem}
\label{gen_c_shift_properties}
Given a GCS matrix $C_{k,n}\in\mathbb{C}^{n\times n}$, the function $\tau_k:\mathbb{Z}^+\to \{1,\ldots,n-1\}$ determined by 
\begin{equation}
\tau_k(t)=
\left\{
\begin{array}{l}
t, ~~ 1\leq t< k-1\\
k-1+(t-k+1)\mathrm{mod}~T, ~~ t\geq k-1
\end{array}
\right.
\label{tau_def}
\end{equation}
with $T=n-k+1$, satisfies the constraints $C_{k,n}^t\hat{e}_{1,n}=\hat{e}_{1+\tau_k(t),n}$, for $t\in\mathbb{Z}^+$.
\end{theorem}
\begin{proof}
Given two positive integers $k,n$ such that $k\leq n$, and any integer $s\geq k-1$ one can represent $s$ in the form
\begin{equation}
s=k-1+r+m(n-k+1),
\label{integer_decomposition}
\end{equation}
for some integers $0\leq r\leq n-k$ and $m\geq 0$. By \eqref{Ck_min_poly} and \eqref{integer_decomposition} it can be seen that for any $t\geq k-1$ there are integers $r,m\geq 0$ such that $r\leq n-k$ and
\begin{eqnarray}
C^t_{k,n}&=&C_{k,n}^{k-1+r+m(n-k+1)}\nonumber\\
&=&C_{k,n}^r(C_{k,n}^{(n-k+1)})^mC_{k,n}^{k-1}=C_{k,n}^rC_{k,n}^{k-1}.\nonumber\\
\label{c_shift_per_constraint_1}
\end{eqnarray}
By \eqref{c_shift_per_constraint_1} we will have that for any $t\in\mathbb{Z}^+$, $C_{k,n}^{t}=C_{k,n}^{\tau_k(t)}$, where $\tau_k(t)$ is defined \eqref{tau_def}. By \cref{shift_property}, \eqref{Ck_min_poly}, \eqref{tau_def} and \eqref{c_shift_per_constraint_1} we will have that for $t\in \mathbb{Z}^+$, $C_{k,n}^t\hat{e}_{1,n}=C_{k,n}^{\tau_k(t)}\hat{e}_{1,n}=\hat{e}_{1+\tau_k(t),n}$. This completes the proof.
\end{proof}

\subsection{Data-Driven Matrix Control Laws}

Given a data-driven system $\Sigma$ in $\mathbb{C}^n$ together with an orbit determined by the sequence $\{x_t\}_{t\geq 1}\subseteq \Sigma\subseteq \mathbb{C}^n$, and given some integer $1\leq S\ll n$, by a {\em \bf data-driven control law} based on a sample $\{\tilde{x}_t\}_{1\leq t\leq S}\subseteq \{x_t\}_{t\geq 1}$, we will mean a matrix set $\{F_t\}_{1\leq t\leq S-1}$ in $\mathbb{C}^{n\times n}$ such that $F_t\tilde{x}_1=\tilde{x}_{t+1}$ and $\mathrm{rank}(F_t)\leq S$ for each $1\leq t\leq S-1$.

We will say that an orbit $\{x_t\}_{t\geq 1}\subseteq \Sigma\subseteq \mathbb{C}^{n}$ of a data-driven system $\Sigma$ is {\bf \em approximately eventually periodic} (AEP), if for any $\varepsilon>0$ there are two integers $T'\geq 1$, $s'\geq 0$, and a vector sequence $\{\tilde{x}_t\}_{t\geq 1}\subset \mathbb{C}^{n}$ of non-zero vectors such that.
\begin{equation}
\left\{
\begin{array}{l}
\|\tilde{x}_t-x_t\|_2\leq \varepsilon\\
\tilde{x}_{t+s'+T'}=\tilde{x}_{t+s'}
\end{array}
\right., t\in\mathbb{Z}^+
\label{index_def}
\end{equation}
Let us consider the smallest integers $0\leq s\leq s'$ and $1\leq T\leq T'$, for which the relations \eqref{index_def} still hold, the number $s+T+1$ will be called the $\varepsilon$-index of the orbit $\{x_t\}_{t\geq 1}$, and will be denoted by $\mathrm{ind}_\varepsilon(\{x_t\})$.

Given an orbit $\{x_t\}_{t\geq 1}$ of a data-driven system $\Sigma$. We say that the control law $\{F_{t}\}_{1\leq t\leq S-1}$ based on some sample $\{\tilde{x}_{t}\}_{1\leq t\leq S}$ is {\em \bf meaningful}, if it (approximately) mimics the dynamical behavior of $\{x_t\}_{t\geq 1}$ {\em nearby} $\tilde{x}_1\in\Sigma$, in particular, if each matrix $F_{s}$ resembles the spectral (or pseudospectral) behavior of the connecting matrix $\mathbb{K}$ (in the sense of \cite[\S2]{DMD_Schmid} and \cite[\S2.2]{DMD_Kutz}) of the data-driven system $\Sigma$ under study, that {\em approximately} satisfies the equations $\mathbb{K}\tilde{x}_t=\tilde{x}_{t+1}$, $1\leq t\leq S-1$.

The matrix control laws for the orbit's data $\{x_t\}_{t\geq 1}\subset \mathbb{C}^n$ of a data-driven system $\Sigma$ are also related to the connecting operator $\mathbb{K}$, for some given orbit's sampled data $\{\tilde{x}_t\}_{1\leq t\leq S}$ with $S\ll n$, by the constrained matrix equations.
\begin{equation}
\left\{
\begin{array}{l}
F_t\tilde{x}_1=\mathbb{K}\tilde{x}_{1+t}\\
\mathrm{rank}(F_t)\leq S
\end{array}
\right., 1\leq t\leq S-1
\label{connecting_matrix_equations}
\end{equation}

Given an orbit $\{x_t\}_{t\geq 1}\subset \mathbb{C}^n$ of a data-driven system $\Sigma$, we say that $\{x_t\}_{t\geq 1}$ is {\bf \em cyclically controlled}, if there is a meaningful control law $\{F_t\}_{1\leq t\leq S-1}\subset\mathbb{C}^{n\times n}$ based on some sample $\{\tilde{x}_t\}_{1\leq t\leq S}\subseteq \{x_t\}_{t\geq 1}$ such that for each $t\geq 1$, there is $1\leq \tau(t)\leq S$ such that $x_{t+1}=F_{\tau(t)}x_t$, and if in addition, there are $\alpha>0$ and matrices $\mathcal{K},\mathcal{L}\in\mathbb{C}^{n}$, $Z_m\in\mathbb{C}^{n\times m}$ such that
\begin{equation}
\left\{
\begin{array}{l}
F_{\tau(t)}=\alpha\mathcal{K}Z_m C_{k,m}^{\tau(t)} Z_m^\ast \mathcal{L}, ~~ t\in\mathbb{Z}^+\\
\mathcal{K}^2=\mathcal{K}^\ast=\mathcal{K}\\
\mathcal{L}^2=\mathcal{L}^\ast=\mathcal{L}
\end{array}
\right.
\label{cyclic_matrix_decomposition_1}
\end{equation}
where $C_{k,m}$ is the GCS matrix determined by \eqref{c_shift_first_rep}. One can notice that \eqref{cyclic_matrix_decomposition_1} produces a representation for the evolution of $\{x_t\}_{t\geq 1}$ in terms of the following closed-loop discrete-time system.
\begin{equation}
\tilde{\Sigma}:\left\{
\begin{array}{l}
\hat{x}_1=\alpha\mathcal{L}x_1\\
\hat{x}_{t+1}=Z_m C_{k,m}^{\tau(t)} Z_m^\ast \hat{x}_t, ~~ t\in\mathbb{Z}^+\\
x_{t}=\mathcal{K}\hat{x}_t
\end{array}
\right.
\label{cyclic_matrix_decomposition_2}
\end{equation}

We call the system $\tilde{\Sigma}$ described by \eqref{cyclic_matrix_decomposition_2} a cyclic finite state approximation (CFSA) of the data-driven system $\Sigma$ based on the sample $\{\tilde{x}_t\}_{1\leq t\leq S}\subseteq\Sigma$, and the GCS $C_{k,m}$ in \eqref{cyclic_matrix_decomposition_2} will be called the GCS {\bf \em factor} of the CFSA $\tilde{\Sigma}$ based on $\{\tilde{x}_t\}_{1\leq t\leq S}$. 

Because of \eqref{cyclic_matrix_decomposition_1} and \eqref{cyclic_matrix_decomposition_2}, one can interpret the computation of a CFSA $\tilde{\Sigma}$ for a given data-driven system $\Sigma$, as a structure preserving matrix approximation problem for structured matrices determined by \eqref{cyclic_matrix_decomposition_1}. The evolution of the CFSA $\tilde{\Sigma}$ of a given data-driven system $\Sigma$ that is described by \eqref{cyclic_matrix_decomposition_2}, can also be interpreted in terms of graphs like the one shown in \cref{finite_control_graph}.

\begin{figure}[!h]
\begin{center}
\includegraphics[scale=.4]{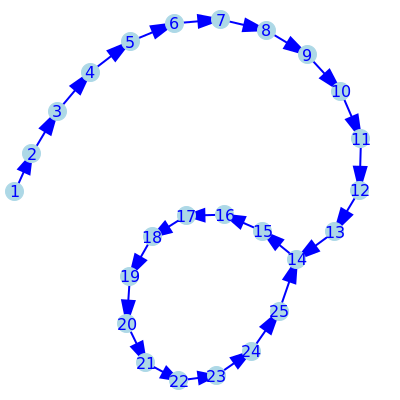}
\end{center}
\caption{Evolution diagram for a $25$-state CFSA $\hat{\Sigma}$ of an AEP data-driven system $\Sigma$.}
\label{finite_control_graph}
\end{figure}

\subsubsection{Predictive finite state approximations}
\label{main_results}

We will study the existence of cyclic finite-state approximations of the form \eqref{cyclic_matrix_decomposition_2} for approximately eventually periodic orbits of data-driven systems. 

Let us consider an orbit's sample $\{x_{t}\}_{1\leq t\leq N}$ from a data-driven system $\Sigma$, and consider the connecting matrix $\mathbb{K}$ determined by dynamic mode decomposition (in the sense of \cite[\S2]{DMD_Schmid}), we will have that $\mathbb{K}$ is an approximate solution to the matrix equation
\begin{equation}
\mathbb{K}W^{(0)}=W^{(1)}=W^{(0)}\mathbb{S}
\label{DMD_rep}
\end{equation}
with $W^{(0)}=[x_1 ~ \cdots ~ x_{N-1}]$, $W^{(1)}=[x_2 ~ \cdots ~ x_{N}]$, and where the companion matrix $\mathbb{S}$ in \eqref{DMD_rep} can be approximated by the least squares solution to the matrix equation $W^{(0)}\mathbb{S}=W^{(1)}$.

\begin{theorem}
\label{c_control_main_result}
Given $\varepsilon>0$, an AEP orbit $\{x_t\}_{t\geq 1}$ of a data-driven system $\Sigma\subseteq \mathbb{C}^n$ has a CFSA whenever $2~\mathrm{ind}_\varepsilon(\{x_t\})\leq n$.
\end{theorem}

\begin{proof}
Given $\varepsilon>0$, and any AEP orbit $\{x_t\}_{t\geq 1}$ of a data-driven system $\Sigma$ such that $2~\mathrm{ind}_\varepsilon(\{x_t\})\leq n$. We will have that there are integers 
$T\geq 1$, $s\geq 0$, and a sequence $\{\tilde{x}_t\}_{t\geq 1}\subset \mathbb{C}^{n}$ of non-zero vectors such that $\|\tilde{x}_t-x_t\|_2\leq \varepsilon$, $\tilde{x}_{t+s+T}=\tilde{x}_{s+t}$ for each $t\geq 1$, with $2(s+T+1)=2~\mathrm{ind}_\varepsilon(\{x_t\})\leq n$.

Since $\tilde{x}_{t+s+T}=\tilde{x}_{t+s}$ for each $t\geq 1$, one can find a sample $\{\tilde{x}_t\}_{1\leq t\leq s+T}\subseteq\{\tilde{x}_t\}_{t\geq 1}$ such that  $\{\tilde{x}_t\}_{t\geq 1}\subseteq \{\tilde{x}_t\}_{1\leq t\leq s+T}$. Let us consider the snapshot matrix $X\in \mathbb{C}^{n\times (s+T)}$ determined by the expression. 
\begin{equation}
X=\begin{bmatrix}
\tilde{x}_1 & \tilde{x}_2 & \cdots & \tilde{x}_s & \cdots & \tilde{x}_{s+T}
\end{bmatrix}
\label{snapshot_matrix_rep}
\end{equation}
By singular value decomposition properties we will have that $X$ can be decomposed in the form,
\begin{equation}
X=USV
\label{snapshot_matrix_svd}
\end{equation}
where $U\in \mathbb{C}^{n\times (s+T)}$ and $V\in \mathbb{C}^{(s+T)\times (s+T)}$ satisfy 
$U^\ast U=\mathbf{1}_{s+T}=V^\ast V$ and where $S=[s_{ij}]\in \mathbb{R}^{(s+T)\times (s+T)}$ is a diagonal matrix with non-negative entries, such that.
\begin{equation}
\left\{
\begin{array}{l}
s_{11}>0\\
s_{11}\geq s_{jj}, 1\leq j\leq s+T
\end{array}
\right.
\end{equation}

Since $2(s+T+1)=2~\mathrm{ind}_\varepsilon(\{x_t\})\leq n$, by Gram-Schmidt orthogonalization theorem we will have that there is $W\in \mathbb{C}^{n\times (s+T)}$ such that.
\begin{equation}
\left\{
\begin{array}{l}
W^\ast W=\mathbf{1}_{s+T} \\
W^\ast U=\mathbf{0}_{s+T}
\end{array}
\right.
\label{W_constraints}
\end{equation}

Let us define a diagonal matrix $\hat{T}=[t_{ij}]$ in $\mathbb{R}^{(s+T)\times (s+T)}$ such that.
\begin{equation}
t_{jj}=\sqrt{1-\left(\frac{s_{jj}}{s_{11}}\right)^2}, 1\leq j\leq n
\label{T_def}
\end{equation}
Let us set. 
\begin{equation}
\left\{
\begin{array}{l}
\hat{X}=(1/s_{11})USW\\
\hat{Y}=V\hat{T}W\\
\hat{Z}=\hat{X}+\hat{Y}
\end{array}
\right.
\label{main_factors_def}
\end{equation}
By \eqref{snapshot_matrix_svd}, \eqref{main_factors_def} and by direct computation we will have that $X=s_{11}\hat{X}$ and $\hat{Z}^\ast \hat{Z}=\mathbf{1}_{s+T}$. Let us consider a representation for the matrix $\hat{Z}$ of the form $\hat{Z}=[\hat{z}_1 ~ \hat{z}_2 ~ \cdots ~ \hat{z}_{s+T}]$, and let us set.
\begin{equation}
\left\{
\begin{array}{l}
\mathcal{K}=UU^\ast\\
\mathcal{T}=\hat{Z}C_{s+1,s+T}\hat{Z}^\ast\\
\mathcal{L}=\hat{z}_1\hat{z}_1^\ast
\end{array}
\right.
\label{cyclic_decomposition_1}
\end{equation}

By \cref{gen_c_shift_properties} and by direct computation we will have that.
\begin{eqnarray}
\mathcal{T}^t\hat{z}_1&=&\hat{Z}C_{s+1,s+T}^t\hat{Z}^\ast\hat{z}_1=\hat{Z}C_{s+1,s+T}^t\hat{e}_{j,s+T}\nonumber\\
&=&\hat{Z}\hat{e}_{1+\tau_{s+1}(t),s+T}=\hat{z}_{1+\tau_{s+1}(t)}
\label{cyclic_decomposition_2}
\end{eqnarray}
By \eqref{main_factors_def} and \eqref{W_constraints} we will have that. 
\begin{equation}
\hat{z}_1^\ast \tilde{x}_1=\frac{\tilde{x}_{1}^\ast \tilde{x}_{1}}{s_{11}}=\frac{\|\tilde{x}_1\|_2^2}{s_{11}}>0
\label{alpha_def}
\end{equation}
One can now easily verify that.
\begin{equation}
\left\{
\begin{array}{l}
\mathcal{K}^2=\mathcal{K}=\mathcal{K}^\ast\\
\mathcal{L}^2=\mathcal{L}=\mathcal{L}^\ast\\
\mathcal{K}\hat{Z}=\hat{X}\\
\mathcal{L}\tilde{x}_1=(\hat{z}_1^\ast \tilde{x}_1)\hat{z}_1
\end{array}
\right.
\label{cyclic_decomposition_3}
\end{equation}

Let us set $\alpha=s_{11}^2/\|\tilde{x}_1\|_2^2$, we will have that $\alpha>0$. By \eqref{c_shift_per_constraint_1}, \eqref{cyclic_decomposition_1}, \eqref{cyclic_decomposition_2} and \eqref{alpha_def} we will have that for each $t\in \mathbb{Z}^+$.
\begin{eqnarray}
\alpha\mathcal{K}\mathcal{T}^t\mathcal{L}\tilde{x}_1&=&\mathcal{K}\frac{s_{11}}{(\hat{z}_1^\ast \tilde{x}_1)}\mathcal{T}^t\mathcal{L}\tilde{x}_1\nonumber\\
&=&\mathcal{K}s_{11}\mathcal{T}^t\hat{z}_1=s_{11}\mathcal{K}\hat{z}_{1+\tau_{s+1}(t)}\nonumber\\
&=&\frac{s_{11}}{s_{11}}\tilde{x}_{1+\tau(t)}=\tilde{x}_{1+\tau_{s+1}(t)}
\label{cyclic_decomposition_4}
\end{eqnarray}

Let us set $F_t=\alpha\mathcal{K}\hat{Z} C_{s+1,s+T}^{\tau_{s+1}(t)}\hat{Z}^\ast\mathcal{L}$, with $\tau_{s+1}$ defined by \eqref{tau_def}. By \eqref{cyclic_decomposition_2},\eqref{cyclic_decomposition_3} and \eqref{cyclic_decomposition_4}, we will have that $\{F_t\}_{t\geq 1}$ is a meaningful matrix control law for $\{\tilde{x}_t\}_{t\geq 1}$, and the evolution of $\{\tilde{x}\}_{t\geq 1}$ will be controlled by the closed-loop system.
\begin{equation}
\tilde{\Sigma}:\left\{
\begin{array}{l}
\hat{x}_1=\alpha\mathcal{L}\tilde{x}_1\\
\hat{x}_{t+1}=\hat{Z} C_{s+1,s+T}^{\tau_{s+1}(t)} \hat{Z}^\ast \hat{x}_t, ~~ t\in\mathbb{Z}^+\\
\tilde{x}_{t}=\mathcal{K}\hat{x}_t
\end{array}
\right.
\end{equation}
This completes the proof.
\end{proof}

By combining \eqref{cyclic_matrix_decomposition_1} and \eqref{DMD_rep}, one can derive the following classification/approximation result.

\begin{theorem}
\label{main_classification_theorem}
Given some orbit's sampled data $\{\tilde{x}_t\}_{1\leq t\leq N}$ from a data-driven system $\Sigma\subseteq \mathbb{C}^n$ with $2N\leq n$, we will have that the integer index $k$ of the GCS factor $C_{k,N-1}$ of the CFSA $\tilde{\Sigma}$ based on the sample, is determined by $k={\arg \min}_{1\leq j\leq N-1}\|\tilde{x}_j-\tilde{x}_N\|_2$.
\end{theorem}
\begin{proof}
Let us set $W^{(0)}=[\tilde{x}_1 ~ \cdots ~ \tilde{x}_{N-1}]$, $W^{(1)}=[\tilde{x}_2 ~ \cdots ~ \tilde{x}_{N}]$, and let us write $\hat{c}_{N-1}$ to denote the $(N-1)$-column of $C_{k,N-1}$. By changing basis and reordering, if necessary, one can think of the GCS factor $C_{k,N-1}$ for the CFSA $\tilde{\Sigma}$ as a least squares solution to the matrix equation
\begin{equation}
W^{(1)}=W^{(0)}C_{k,N-1}
\end{equation}
where only $\hat{c}_{N-1}$ needs to be determined, and is constrained by \eqref{c_shift_first_rep} to satisfy $\hat{c}_{N-1}=\hat{e}_{k,n}$ for some $1\leq k\leq N-1$. This in turn implies that $k={\arg \min}_{1\leq j\leq N-1} \|W^{(0)}\hat{e}_{j,n}-\tilde{x}_{N}\|_2={\arg \min}_{1\leq j\leq N-1} \|\tilde{x}_{j}-\tilde{x}_{N}\|_2$. This completes the proof.
\end{proof}

\section{Algorithm}
\label{section_algorithms}
We have that the previous theoretical techniques can be implemented to derive finite-state approximation algorithms for approximately eventually periodic data-driven systems, that can be described using the following transition diagram.

\vspace*{1pc}

\tikzstyle{block} = [draw, fill=white, rectangle, 
    minimum height=3em, minimum width=6em]
\tikzstyle{sum} = [draw, fill=white, circle, node distance=1cm]
\tikzstyle{input} = [coordinate]
\tikzstyle{output} = [coordinate]
\tikzstyle{pinstyle} = [pin edge={to-,thin,black}]

\begin{tikzpicture}[auto, node distance=2cm,>=latex']

    \node [input, name=input] {};
    \node [block, right of=input] (delay) {$Z^{-1}$};
    \node [block, right of=delay, node distance=3cm] (K) {K};

    \draw [->] (delay) -- node[name=u] {$\hat{x}_t$} (K);
    
    \node [output, right of=K] (output) {};
    \node [block, below of=delay, pin={[pinstyle]below: $\xymatrix{
    \mathbb{C}[C_{k,m}] }$},
            node distance=2cm] (feedback) {$\hat{F}_t$};

    \draw [->] (input) -- node {$\hat{x}_{t+1}$} (delay);
    \draw [->] (K) -- node [name=y] {$x_t$}(output);
    \draw [-] (feedback) -| node[name=input] {$$} 
        node [near end] {$$} (input);
    \draw [->] (u) |- (feedback);    
\end{tikzpicture}

A prototypical algorithm for data-driven CFSA computation is presented below.

\begin{algorithm}[h!]
\caption{Data-driven cyclic finite state approximation algorithm}
\label{alg:main_alg}
\begin{algorithmic}
\STATE{{\bf Data:}\:\:\:{\sc Real number $\varepsilon>0$, State data history:} $\{x_t\}_{1\leq t\leq s+ T+1}$, $s,T\in \mathbb{Z}^+$}
\STATE{{\bf Result:}\:\:\: {{\sc Approximate matrix realizations:} $(\mathcal{K},\mathcal{T},\mathcal{L})\in \mathbb{C}^{n\times n}\times \mathbb{C}^{n\times n}\times \mathbb{C}^{n\times n}$ of $\tilde{\Sigma}$}}
\begin{enumerate}
\STATE{Set $m=s+T$\;}
\STATE{Get/Compute sample $\{x_t\}_{1\leq t\leq m+1}$ from $\Sigma$\;}
\STATE{Compute $k={\arg\min}_{1\leq j\leq {m}}\|x_{m+1}-x_{j}\|_2$\;}
\STATE{Compute the SVD $\mathbb{V}S\mathbb{W}=[x_1~\cdots~x_{m}]$\;}
\STATE{Compute the matrix $\hat{\mathbb{Z}}_m=[\hat{z}_1~\cdots~\hat{z}_m]$ in \eqref{main_factors_def} for $\{x_t\}_{1\leq t\leq m}$\;}
\STATE{Set $\mathcal{K}=\mathbb{V}\mathbb{V}^\ast$\;}
\STATE{Set $\mathcal{L}=(s_{11}^2/\|\tilde{x}_1\|_2^2)\hat{z}_1\hat{z}_1^\ast$\;}
\STATE{Set $\mathcal{T}=\hat{\mathbb{Z}}_m C_{k,m}\hat{\mathbb{Z}}_m^\ast$\;}
\end{enumerate}
\RETURN $\{\mathcal{K},\mathcal{T},\mathcal{L}\}$
\end{algorithmic}
\end{algorithm}

\section{Numerical Experiments}
\label{section_experiments}
In this section we will consider the snapshot matrices of approximately periodic and eventually periodic sampled orbits $\Sigma_{P}$ and $\Sigma_{EP}$ in $\mathbb{C}^{10001\times 202}$, respectively, that are generated by discretizations (and perturbations) of generic damped wave models determined by equations of the form.

\begin{equation}
\left\{
\begin{array}{l}
\partial_t^2 \psi=\alpha^2\partial_x^2\psi+\delta\partial_t\psi,~(x,t)\in (0,L)\times (0,\infty)\\
\psi(x,0)=\psi_0(x)\\
\partial_t\psi(x,0)=\psi_1(x)\\
\partial_x\psi(0,t)=\psi(L,t)=0
\end{array}
\right.
\label{wave_model}
\end{equation}
for some $L>0$, with $\alpha,\delta\in\mathbb{R}$. Discretizations of \eqref{wave_model} can be computed using Crank-Nicolson-type finite-difference schemes of the form
\begin{equation}
\left\{
\begin{array}{l}
A\Phi_{t+1}=B\Phi_{t}\\
\Psi_{t}=C\Phi_t
\end{array}
\right.,t\geq 1
\label{discrete_wave_model}
\end{equation}
for some matrices $A,B\in \mathbb{C}^{20002\times 20002}$ and $C\in \mathbb{C}^{10001\times 20002}$, with $A$ invertible and $\Phi_t\in\mathbb{C}^{20002\times 1}$ for each $t\geq 1$. 

Generic models of the form \eqref{wave_model} and \eqref{discrete_wave_model} together with their perturbations, have applications in the simulation of transmission lines for electrical signals in computer networks, power electronics, embedded and power systems, among others. The dynamical behavior forcasting determined by the $201$-state CFSA $\tilde{\Sigma}_{P}$ for the approximately periodic orbit $\Sigma_{P}$ is shown in figure \cref{fig:experiment_1}.

\begin{figure}[!h]
\begin{center}
\includegraphics[scale=.4]{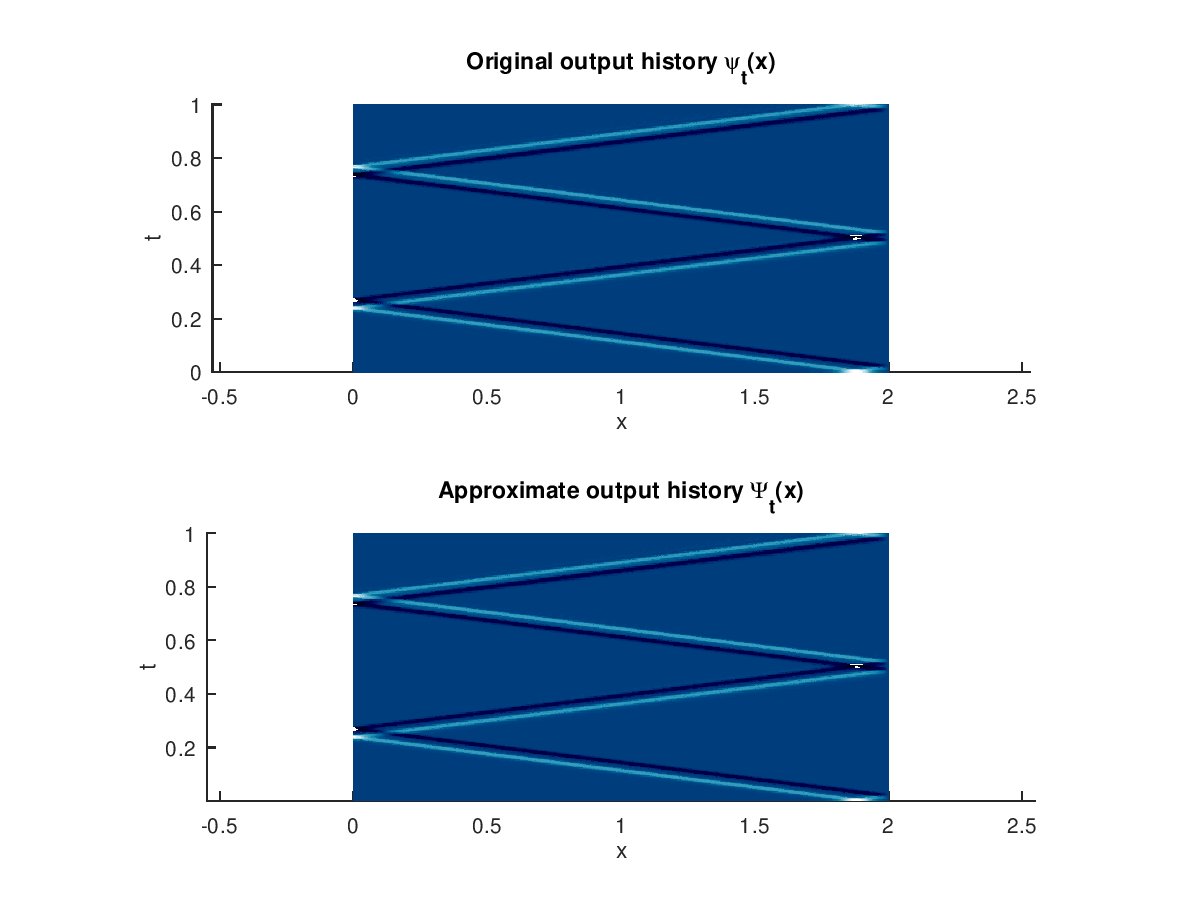}
\end{center}
\caption{A $201$-state approximation $\tilde{\Sigma}_P$ (bottom), based on a sample $\Sigma_P$ from an AEP orbit of a transmission line model (top).}
\label{fig:experiment_1}
\end{figure}

In order to visualize the meaningfulness of the control law of the CFSA $\tilde{\Sigma}_{P}$ of $\Sigma_{P}$ for an approximation error of $O(10^{-12})$, the pseudospectra $\sigma_\varepsilon(\mathbb{S}_{\Sigma_P})$ and $\sigma_\varepsilon(C_{k,201})$, with $0\leq \varepsilon\leq \varepsilon'$ for some fixed $\varepsilon'>0$, for the companion matrix $\mathbb{S}_{\Sigma_P}$ determined by \eqref{DMD_rep}, and the GCS factor $C_{k,201}$ predicted for $\tilde{\Sigma}_{P}$ by \cref{main_classification_theorem} and \cref{alg:main_alg}, are shown in \cref{finite_control_pspectra_1_1}.

\begin{figure}[!h]
\begin{center}
\includegraphics[scale=.4]{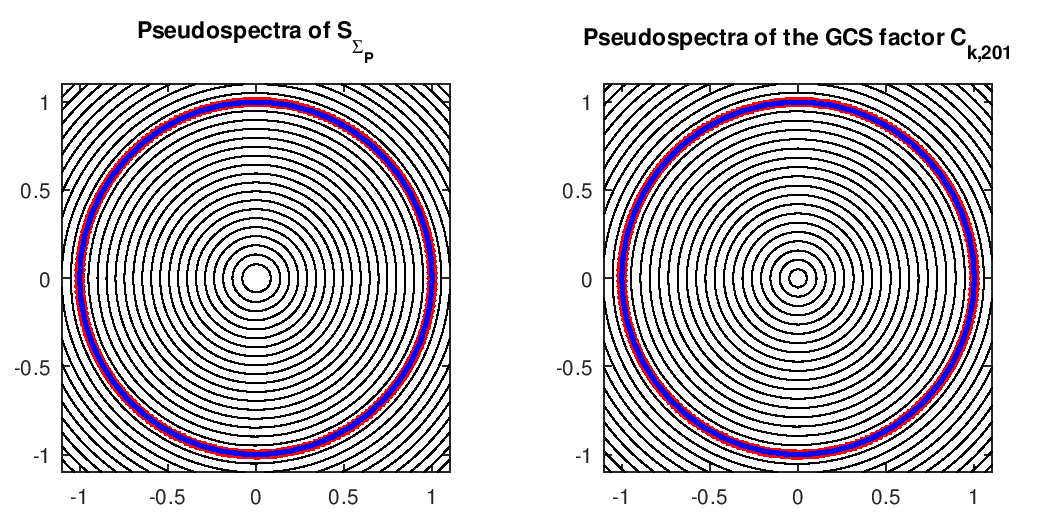}
\includegraphics[scale=.4]{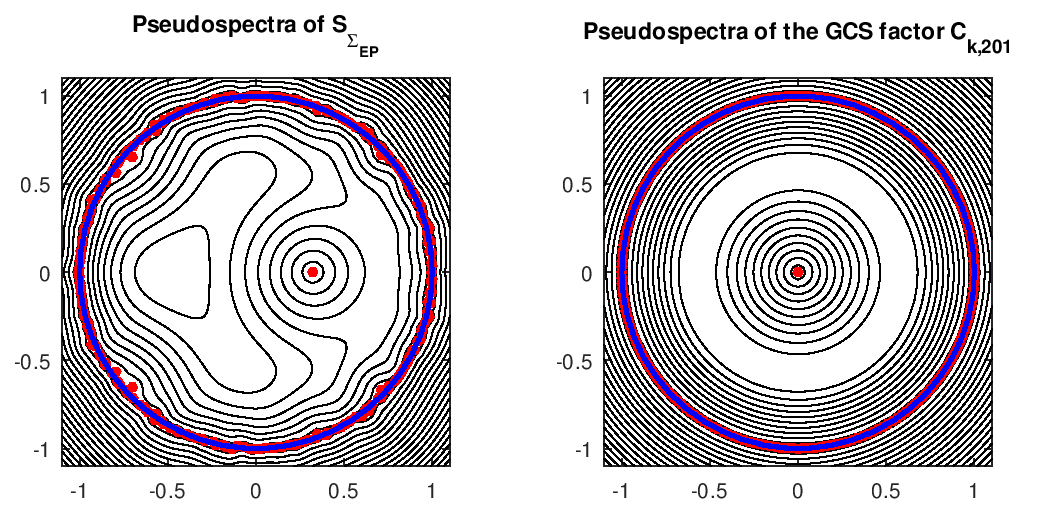}
\end{center}
\caption{Pseudospectra $\sigma_\varepsilon(\mathbb{S}_{\Sigma_P})$ (top-left), $\sigma_\varepsilon(C_{k,201})$ for the predicted GCS factor of $\tilde{\Sigma}_P$ (top-right), $\sigma_\varepsilon(\mathbb{S}_{\Sigma_{EP}})$ (bottom-left), $\sigma_\varepsilon(\hat{C}_{k,201})$ for the predicted GCS factor of $\tilde{\Sigma}_{EP}$ (bottom-right), for $0\leq\varepsilon\leq\varepsilon'$. The black lines represent elements in the pseudospectra, the blue lines represent $\mathcal{S}^1$, and the red dots represent the eigenvalues.}
\label{finite_control_pspectra_1_1}
\end{figure}

If the system determined by \eqref{discrete_wave_model} is perturbed, simulating either perturbations in the original model \eqref{wave_model}, or imprecisions caused by noisy measurements, one can obtain a perturbed AEP orbit represented by the snapshot matrix $\Sigma_{EP}$ together with a dynamical behavior forcasting determined by a $201$-state CFSA $\tilde{\Sigma}_{EP}$, like the ones shown in \cref{fig:experiment_2}.

\begin{figure}[!h]
\begin{center}
\includegraphics[scale=.4]{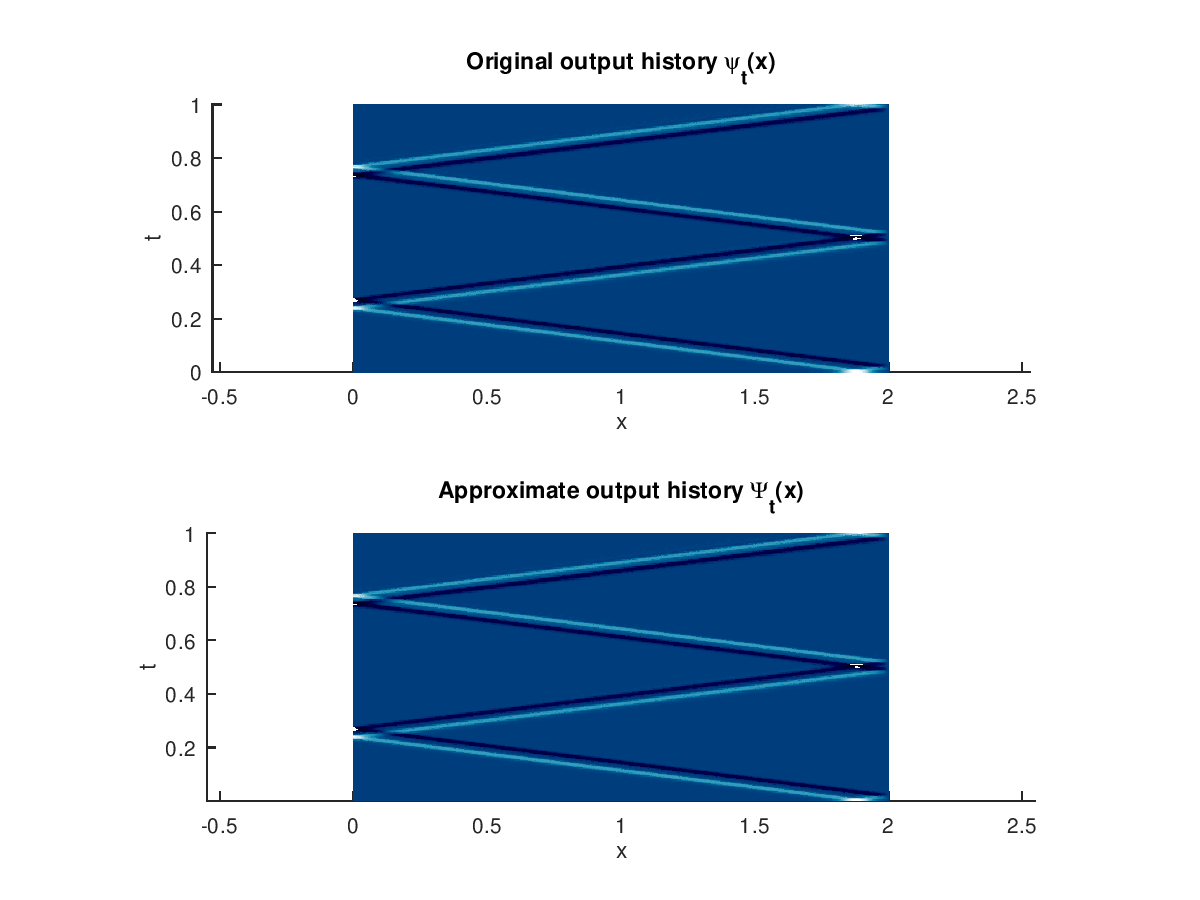}
\end{center}
\caption{A $201$-state approximation $\tilde{\Sigma}_{EP}$ (bottom), based on a sample $\Sigma_{EP}$ from a perturbed AEP orbit of a transmission line model (top).}
\label{fig:experiment_2}
\end{figure}

In order to visualize the meaningfulness of the control law of the CFSA $\tilde{\Sigma}_{EP}$ of $\Sigma_{EP}$, for an approximation error of $O(10^{-12})$ the $\varepsilon$-pseudospectra of the companion matrix $\mathbb{S}_{\Sigma_{EP}}$ and the predicted GCS factor for $\tilde{\Sigma}_{EP}$, are also shown in \cref{finite_control_pspectra_1_1}.

\section{Conclusion and Future Directions}
The results in \S\ref{main_results} allow one to derive computational methods like the one described in \cref{alg:main_alg}, for finite state approximation/forcasting of the dynamical behavior of a data-driven system determined by some data sampled from a set of valid/feasible states. 

Some applications of \cref{alg:main_alg} to data-based artificially intelligent schemes that learn from mistakes, and can be used for model predictive control of industrial processes, will be presented in future communications.

The connections of the results in \S\ref{main_results} to the solution of problems related to controllability and realizability of finite-state systems in classical and quantum information and automata theory, in the sense of  \cite{finite_state_systems,finite_quantum_control_systems,finite_state_machine_approximation,BROCKETT20081}, will be further explored.

\section*{Acknowledgment}

The structure preserving matrix computations needed to implement \cref{alg:main_alg}, were performed in the Scientific Computing Innovation Center ({\bf CICC-UNAH}) of the National Autonomous University of Honduras.

I am grateful with Terry Loring, Marc Rieffel, Marius Junge, Douglas Farenick, Masoud Khalkhali, Alexandru Chirvasitu, Concepci\'on Ferrufino, Leonel Obando, Mario Molina and William F\'unez for several interesting questions and comments, that have been very helpful for the preparation of this document.

\bibliographystyle{IEEEtran}
\bibliography{FredyVides}

\begin{thebibliography}{1}
\providecommand{\url}[1]{#1}
\csname url@samestyle\endcsname
\providecommand{\newblock}{\relax}
\providecommand{\bibinfo}[2]{#2}
\providecommand{\BIBentrySTDinterwordspacing}{\spaceskip=0pt\relax}
\providecommand{\BIBentryALTinterwordstretchfactor}{4}
\providecommand{\BIBentryALTinterwordspacing}{\spaceskip=\fontdimen2\font plus
\BIBentryALTinterwordstretchfactor\fontdimen3\font minus
  \fontdimen4\font\relax}
\providecommand{\BIBforeignlanguage}[2]{{%
\expandafter\ifx\csname l@#1\endcsname\relax
\typeout{** WARNING: IEEEtran.bst: No hyphenation pattern has been}%
\typeout{** loaded for the language `#1'. Using the pattern for}%
\typeout{** the default language instead.}%
\else
\language=\csname l@#1\endcsname
\fi
#2}}
\providecommand{\BIBdecl}{\relax}
\BIBdecl

\bibitem{bookPspectra}
L.~Trefethen and M.~Embree, \emph{Spectra and Pseudospectra: The behavior of
  nonnormal matrices and operators}.\hskip 1em plus 0.5em minus 0.4em\relax
  Princeton University Press, 01 2005.

\bibitem{DMD_Schmid}
S.~P. J., ``Dynamic mode decomposition of numerical and experimental data,''
  \emph{J. Fluid Mech.}, vol. 656, pp. 5--28, 2010.

\bibitem{DMD_Kutz}
B.~S.~L. Proctor J.~L. and K.~J. N., ``Dynamic mode decomposition with
  control,'' \emph{SIAM J Appl. Dyn. Syst.}, vol.~15, no.~1, pp. 142--161,
  2016.

\bibitem{finite_state_systems}
R.~{Brockett} and A.~{Willsky}, ``Finite group homomorphic sequential system,''
  \emph{IEEE Transactions on Automatic Control}, vol.~17, no.~4, pp. 483--490,
  August 1972.

\bibitem{finite_quantum_control_systems}
A.~M. {Bloch}, R.~W. {Brockett}, and C.~{Rangan}, ``Finite controllability of
  infinite-dimensional quantum systems,'' \emph{IEEE Transactions on Automatic
  Control}, vol.~55, no.~8, pp. 1797--1805, Aug 2010.

\bibitem{finite_state_machine_approximation}
D.~C. {Tarraf}, ``An input-output construction of finite state $\rho/\mu$
  approximations for control design,'' \emph{IEEE Transactions on Automatic
  Control}, vol.~59, no.~12, pp. 3164--3177, Dec 2014.

\bibitem{BROCKETT20081}
R.~W. Brockett, ``Reduced complexity control systems,'' \emph{IFAC Proceedings
  Volumes}, vol.~41, no.~2, pp. 1 -- 6, 2008, 17th IFAC World Congress.

\end{thebibliography}

\end{document}